\documentclass[11pt]{article}
\usepackage{amsmath,amssymb,amsthm}
\usepackage{amscd}
\newtheorem{theorem}{Theorem}
\newtheorem{lemma}[theorem]{Lemma}
\newtheorem{proposition}[theorem]{Proposition}

\theoremstyle{plain}

\theoremstyle{definition}
\newtheorem{definition}[theorem]{Definition}
\newtheorem{remark}[theorem]{Remark}
\newtheorem{example}[theorem]{Example}

\renewcommand{\labelenumi}{\textup{(\theenumi)}}

\newcommand{\Homeo}{\operatorname{Homeo}}
\newcommand{\id}{\operatorname{id}}
\newcommand{\Ker}{\operatorname{Ker}}
\newcommand{\Ad}{\operatorname{Ad}}

\def\det{{{\operatorname{det}}}}

\newcommand{\N}{\mathbb{N}}

\newcommand{\Zp}{{\mathbb{Z}}_+}

\newcommand{\C}{\mathcal{C}}

\title{Coded equivalence of one-sided topological Markov shifts}
\author{Kengo Matsumoto \\
Department of Mathematics \\
Joetsu University of Education \\
Joetsu, 943-8512, Japan
}

\begin{document}
\maketitle

\date{}

\def\det{{{\operatorname{det}}}}

\begin{abstract}
We introduce a notion of coded equivalence in one-sided topological Markov shifts.
The notion is inspired by coding theory.
One-sided topological conjugacy implies coded equivalence.
We will show that coded equivalence implies  continuous orbit equivalence of
one-sided topological Markov shifts.
\end{abstract}

2020{\it Mathematics Subject Classification}:
 Primary 94A45, 68Q45; Secondary 37B10,  68R15, 37A55.



\def\OA{{{\mathcal{O}}_A}}
\def\OB{{{\mathcal{O}}_B}}
\def\FA{{{\mathcal{F}}_A}}
\def\FB{{{\mathcal{F}}_B}}
\def\DA{{{\mathcal{D}}_A}}
\def\DB{{{\mathcal{D}}_B}}
\def\Ext{{{\operatorname{Ext}}}}
\def\Max{{{\operatorname{Max}}}}
\def\Per{{{\operatorname{Per}}}}
\def\PerB{{{\operatorname{PerB}}}}
\def\Homeo{{{\operatorname{Homeo}}}}
\def\HA{{{\frak H}_A}}
\def\HB{{{\frak H}_B}}
\def\HSA{{H_{\sigma_A}(X_A)}}
\def\Out{{{\operatorname{Out}}}}
\def\Aut{{{\operatorname{Aut}}}}
\def\Ad{{{\operatorname{Ad}}}}
\def\Inn{{{\operatorname{Inn}}}}
\def\det{{{\operatorname{det}}}}
\def\exp{{{\operatorname{exp}}}}
\def\cobdy{{{\operatorname{cobdy}}}}
\def\Ker{{{\operatorname{Ker}}}}
\def\ind{{{\operatorname{ind}}}}
\def\id{{{\operatorname{id}}}}
\def\supp{{{\operatorname{supp}}}}
\def\cod{{{\operatorname{code}}}}
\def\Sco{{{\operatorname{Sco}}}}
\def\Act{{{\operatorname{Act}_{\DA}(\mathbb{R},\OA)}}}
\def\RepOA{{{\operatorname{Rep}(\mathbb{R},\OA)}}}
\def\RepDA{{{\operatorname{Rep}(\mathbb{R},\DA)}}}
\def\U{{{\operatorname{U}}}}

\bigskip

In \cite{MaPacific}, 
the author introduced a notion of continuous orbit equivalence of one-sided topological Markov shifts.
The definition of the equivalence relation was primary inspired by orbit equivalence theory 
of minimal homeomorphisms on Cantor sets 
established by Giordano--Putnam--Skau \cite{GPS}(cf. \cite{HPS}, etc.).
 Through the studies of classifications of continuous orbit equivalence of one-sided topological Markov shifts,
several interesting relationships with other areas of mathematics, 
$C^*$-algebras, groupoids, infinite discrete groups, etc.
  have been clarified (cf. \cite{MMKyoto}, \cite{MMGGD}, etc. ).
As a result, H. Matui and the author have succeeded to classify 
irreducible  one-sided topological Markov shifts under continuous orbit equivalence (\cite{MMKyoto}). 
However, there is no known any systematic method 
to give rise to continuous orbit equivalence of one-sided topological Markov shifts. 
 In this paper, we introduce a notion of {\it coded equivalence}\/ in one-sided topological Markov shifts.
 The definition of the coded equivalence is inspired by coding theory of formal language theory. 
Although relationship between symbolic dynamics and coding theory has been studied by many authors, for example
\cite{Beal}, \cite{BH},  \cite{BH2}, \cite{FieD1}, \cite{FieD2}, \cite{FieF}, etc,
the coded equivalence treated in this paper has not seen in any other papers than this.

We will study the coded equivalence from the view point of symbolic dynamical systems. 
It is well-known that topological conjugacy of symbolic dynamical systems is given by 
a sliding block code (cf. \cite{Kitchens}, \cite{LM}).
We will also introduce a notion of {\it moving block code}\/, that is a generalization of sliding block code. 
We will then see that one-sided topological conjugacy implies coded equivalence.
As a main result of the paper, we will show that the coded equivalence implies  continuous orbit equivalence of
one-sided topological Markov shifts (Theorem \ref{thm:main}).
We therefore know a close relationship between coding theory and continuous orbit equivalence theory of one-sided topological Markov shifts. 
Several examples of coded equivalent topological Markov shifts will be presented.

We will first provide several terminology and notation.
Let us denote by $\N$ and $\Zp$ the set of positive integers 
and the set of nonnegative integers, respectively.
Let $A =[A(i,j)]_{i,j=1}^N$ be an irreducible non permutation matrix over $\{0,1\}$.
Let us denote by $\Sigma_A$ the set $\{1,2,\dots, N\}$.
Let $X_A$ be the set of right one-sided infinite sequences $(x_n)_{n \in \N}$ of $\Sigma_A$   
such that $A(x_n, x_{n+1}) =1$ for all $n \in \N$.
The set $X_A$ is endowed with its product topology so that it is a homeomorphic to a Cantor discontinuum. 
It has a natural shift operation $\sigma_A$ defined by 
$\sigma_A((x_n)_{n\in \N}) = (x_{n+1})_{n\in \N}$.
The topological dynamical system $(X_A, \sigma_A)$ 
is called the one-sided topological Markov shift defined by the matrix $A$. 
The space $X_A$ is called the shift space of $(X_A,\sigma_A)$.
Let us denote by $B_k(X_A)$ the set of admissible words of $X_A$ with length $k$. 
We put $B_*(X_A) = \cup_{k=0}^\infty B_k(X_A)$, where $B_0(X_A)$ denotes the empty word.
For a word $\mu =(\mu_1, \dots,\mu_m) \in B_m(X_A)$,
let us denote by $U_\mu$ the cylinder set
$$
U_\mu =\{ (x_n)_{n\in \N} \in X_A\mid x_1 = \mu_1, \dots, x_m =\mu_m\}.
$$

\medskip

A {\it code}\/ $\C$ of $X_A$ is a nonempty subset $\C \subset B_*(X_A)$
such that for any equality
\begin{equation*}
\omega(i_1) \omega(i_2) \cdots \omega(i_k) 
= 
\omega(j_1) \omega(j_2) \cdots \omega(j_n) 
 \end{equation*} 
of words with
$\omega(i_1), \omega(i_2), \cdots, \omega(i_k), 
\omega(j_1), \omega(j_2), \cdots, \omega(j_n) 
\in \C$,
one has
\begin{equation*}
n =k \qquad
\text{ and }
\qquad
\omega(i_m) = \omega(j_m), \quad
m=1,2, \dots,n.
\end{equation*} 
A {\it prefix code }
is a code such that 
no word in it can be the beginning of another
( cf. \cite{BP}, \cite{BH}).
For example, let $A = 
\begin{bmatrix}
1 & 1 \\
1 & 1
\end{bmatrix}
$
and
$\Sigma_A =\{1,2\}$.
Consider
$\C =\{ 1, 21, 22\}$ and $\C' =\{ 1,12,22\}$.
We see that $\C$ is a prefix code,
whereas $\C'$ is a code but not a prefix code. 

The following lemma is easy to prove.
\begin{lemma}
Let 
$\C = \{ \omega(1),\omega(2), \dots,\omega(M)\} \subset B_*(X_A)$
be a finite set of admissible words of $X_A$.
Then 
$\C$ is a prefix code if and only if
$U_{\omega(i)} \cap U_{\omega(j)} =\emptyset$ for $i\ne j$.
\end{lemma}
For a prefix code
$\C = \{ \omega(1), \omega(2),\dots,\omega(M)\} \subset B_*(X_A)$,
we denote by
$\Sigma_{A(\C)}$ the set $\{1, 2, \dots, M\}$.
Let us denote by $\ell(i)$ the length of the word 
$\omega(i), i\in\Sigma_{A(\C)}.$
The word $\omega(i)$ is written
$$
\omega(i) =\omega_1(i)\omega_2(i)\cdots\omega_{\ell(i)}(i) \in B_{\ell(i)}(X_A)
\quad \text{ for some } \quad
\omega_j(i) \in \Sigma_A.
$$
For the word $\omega(i)$, define $\sigma_A(\omega(i)) \in  B_{\ell(i)-1}(X_A)$
by setting
$$
\sigma_A(\omega(i) ) 
 = \omega_2(i) \omega_3(i) \cdots \omega_{\ell(i)}(i). 
$$
We will introduce a notion of right Markov code in the following way.
\begin{definition}\label{def:Markovcode}
Let $\C=\{ \omega(1),\omega(2),\dots, \omega(M) \} \subset B_*(X_A)$
be a prefix code.
The prefix code $\C$ is called  {\it a right Markov code}\/ for $(X_A,\sigma_A)$
if it satisfies the following three conditions:
\begin{enumerate}
\renewcommand{\theenumi}{\roman{enumi}}
\renewcommand{\labelenumi}{\textup{(\theenumi)}}
\item (unique factorization)
For any $\gamma \in B_*(X_A)$,
there exists a word $\eta \in B_*(X_A)$ such that 
$\gamma \eta \in B_*(X_A)$, and 
there exists a unique finite sequence 
$(i_1, i_2,\dots,i_k) \in {(\Sigma_{A(\C)})}^k$ such that 
\begin{equation*}
\gamma \eta = \omega(i_1) \omega(i_2) \cdots \omega(i_k).
\end{equation*} 
\item (shift invariance)
There exists $L \in \N$ such that for any 
$i_1, i_2,\dots, i_L \in \Sigma_{A(\C)}$
with
$ \omega(i_1) \omega(i_2) \cdots \omega(i_L)\in B_*(X_A)$,
there exists 
$j_1, j_2,\dots, j_k \in \Sigma_{A(\C)}$
 such that 
\begin{equation}\label{eq:shiftinvariance}
\sigma_A(\omega(i_1) )\omega(i_2) \cdots \omega(i_L)
= \omega(j_1) \omega(j_2) \cdots \omega(j_k),
\end{equation} 
 where
$k$ depends on the finite sequence $(i_1, i_2,\dots, i_L)$.
\item (irreducibility)
For any ordered pair $\omega(i), \omega(j) \in \C$,
there exist $n_1, n_2 \dots, n_l \in \Sigma_{A(\C)}$ such that 
\begin{equation*}
\omega(i)\omega(n_1)\omega(n_2) \cdots \omega(n_l) \omega(j) \in B_*(X_A) 
\end{equation*}
\end{enumerate}
\end{definition}
We call (i), (ii) and (iii) unique factorization property, shift equivalence property
and irreducible condition, respectively.
\begin{remark}
In the definition of right Markov code,
we do not necessarily assume that 
$\omega(i) \in \C$ implies $\sigma(\omega(i)) \in \C$.
Hence in the above definition (ii),
the word $\sigma_A(\omega(i))$ does not necessarily belong to the set $\C$ 
when $\omega(i) \in \C$.
\end{remark}

\begin{lemma}\label{lem:factorization}
Let
$\C=\{ \omega(1),\omega(2),\dots, \omega(M) \} \subset B_*(X_A)$
be a prefix code.
The following assertions are 
equivalent.
\begin{enumerate}
\renewcommand{\theenumi}{\roman{enumi}}
\renewcommand{\labelenumi}{\textup{(\theenumi)}}
\item  $\C$ satisfies the unique factorization property (i) of Definition \ref{def:Markovcode}. 
\item For any $x =(x_n)_{n\in \N} \in X_A$
there exists a unique increasing sequence
$1<k_1 <k_2  <\cdots  $ of positive integers
and
$i_1, i_2, \cdots \in \Sigma_{A(\C)}$ such that 
\begin{equation*}
x_{[1,k_1)} =\omega(i_1), \qquad
x_{[k_1,k_2)} =\omega(i_2),\quad \dots, \quad
x_{[k_n,k_{n+1})}=\omega(i_{n+1}), \quad \dots
\end{equation*}
That is,
$x$ has a unique factorization:
$$
x = \omega(i_1)\omega(i_2)\cdots \omega(i_n)\cdots
$$
\end{enumerate}
\end{lemma}
\begin{proof}
(i) $\Longrightarrow$ (ii):
Assume that 
$\C=\{ \omega(1),\omega(2),\dots, \omega(M) \} \subset B_*(X_A)$
satisfies the unique factorization property (i) of Definition \ref{def:Markovcode}.
Put
$K_0 = \max\{ \ell(i) \mid i \in \Sigma_{A(\C)} \},$
where $\ell(i)$ denotes the length of $\omega(i)$.
 Take an increasing sequence 
$1<m_1 <m_2  <\cdots  $ of positive integers
such that 
$n\cdot K_0 < m_n$ for $n \in \N$.
Take an arbitrary $x =(x_n)_{n\in \N} \in X_A$.
For the word
$x_{[1,m_n]} \in B_{m_n}(X_A)$ for each $n \in \N$,
by the unique factorization property (i) of Definition \ref{def:Markovcode},
there exists
$\eta(n) \in B_*(X_A)$ such that 
$x_{[1,m_n]} \eta(n) \in B_*(X_A)$, and 
there exists a unique finite sequence 
$(i_1(n), i_2(n),\dots,i_{k_n}(n)) \in (\Sigma_{A(\C)})^{k_n}$ 
such that 
\begin{equation*}
x_{[1,m_n]} \eta(n)  = \omega(i_1(n)) \omega(i_2(n)) \cdots \omega(i_{k_n}(n)).
\end{equation*} 
Since $\C$ is a prefix code,
we see that 
\begin{align*}
\omega(i_1(1)) &=\omega(i_1(n)) \quad \text{ for all } n \ge 1, \\
\omega(i_2(2)) &=\omega(i_2(n)) \quad \text{ for all } n \ge 2, \\
\cdots & = \cdots \\
\omega(i_\ell(\ell)) &=\omega(i_\ell(n)) \quad \text{ for all } n \ge \ell.
\end{align*}
Take $\xi(n) \in X_A$ such that 
$x_{[1,m_n]} \eta(n) \xi(n) \in X_A$.
Put
$x(n) =x_{[1,m_n]} \eta(n) \xi(n) \in X_A$ for $n \in \N$.
Then $x(n) $ converges to $x$
and
we have
$$
x =\omega(i_1(1))\omega(i_2(2))\cdots \omega(i_\ell(\ell)) \cdots
$$
The factorization is unique, because $\C$ is a prefix code.

(ii) $\Longrightarrow$ (i):
The assertion is obvious.
\end{proof}

For an $N\times N$ irreducible matrix $A=[A(i,j)]_{i,j=1}^N$
with entries in $\{0,1\}$,
an associated directed graph $G_A =(V_A, E_A)$
is defined in the following way.
The vertex set $V_A = \Sigma_A (=\{1,2,\dots, N\})$.
If $A(i,j) =1,$ then a directed edge from the vertex $i$ to the vertex $j$ is defined.
The edge set $E_A$ consists of such edges.
The transition matrix of the graph $G_A$ is the original matrix $A$.
Hence the shift space $X_A$ consists of right infinite sequences of concatenating vertices
in the graph $G_A$.  
\begin{example}

\noindent

{\bf 0.}
Let $A=[A(i,j)]_{i,j=1}^N$
be an $N\times N$ irreducible matrix with entries in $\{0,1\}$.
Define $\C_0 = \Sigma_A$
the set of admissible words of length $1$.
Then $\C_0$ is a right Markov code.
It is called the {\it trivial right Markov code}\/ for $(X_A,\sigma_A)$.

{\bf 1.}
Let
$A_1=
\begin{bmatrix}
1& 1 \\
1&0
\end{bmatrix}.
$
Define
$\C_1 = \{1, 21\}$.
Then $\C_1$ is a right Markov code.
The directed graph $G_{A_1}$ is Figure \ref{fig:fibonacci}.
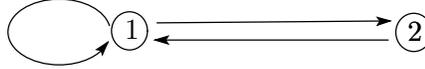
\begin{figure}[htbp]
\begin{center}
\unitlength 0.1in
\begin{picture}( 24.0400,  3.4400)( 17.6000, -9.8400)
%
\special{pn 8}%
\special{ar 2396 806 94 100  0.0000000 6.2831853}%
%
\special{pn 8}%
\special{ar 2028 812 268 172  0.3226621 6.0865560}%
%
\special{pn 8}%
\special{pa 2276 872}%
\special{pa 2282 866}%
\special{fp}%
\special{sh 1}%
\special{pa 2282 866}%
\special{pa 2224 904}%
\special{pa 2248 908}%
\special{pa 2254 930}%
\special{pa 2282 866}%
\special{fp}%
\put(23.7000,-8.5000){\makebox(0,0)[lb]{$1$}}%
%
\put(41.6400,-6.6400){\makebox(0,0)[lb]{}}%
%
\special{pn 8}%
\special{ar 3884 814 94 100  0.0000000 6.2831853}%
%
\put(32.4800,-9.6000){\makebox(0,0)[lb]{}}%
%
\special{pn 8}%
\special{pa 3750 854}%
\special{pa 2530 854}%
\special{fp}%
\special{sh 1}%
\special{pa 2530 854}%
\special{pa 2598 874}%
\special{pa 2584 854}%
\special{pa 2598 834}%
\special{pa 2530 854}%
\special{fp}%
%
\special{pn 8}%
\special{pa 2540 764}%
\special{pa 3760 754}%
\special{fp}%
\special{sh 1}%
\special{pa 3760 754}%
\special{pa 3694 736}%
\special{pa 3708 754}%
\special{pa 3694 776}%
\special{pa 3760 754}%
\special{fp}%
\put(38.5000,-8.6000){\makebox(0,0)[lb]{$2$}}%
\put(38.5000,-8.6000){\makebox(0,0)[lb]{$2$}}%
\end{picture}%
\end{center}
\caption{The directed graph $G_{A_1}$}
\label{fig:fibonacci}
\end{figure}

{\bf 2.}
Let
$A_2=
\begin{bmatrix}
0 & 1& 0 \\
0 & 0& 1 \\
1 & 1& 0
\end{bmatrix}.
$
Define
$\C_2 = \{12, 23, 323, 312\}$.
Then $\C_2$ is a right Markov code.

{\bf 3.}
Let
$A_3=
\begin{bmatrix}
1 & 1& 1 \\
1 & 0& 1 \\
1 & 0& 0
\end{bmatrix}.
$
Define
$\C_3 = \{1,21, 31, 231\}$.
Then $\C_3$ is a right Markov code.

The directed graphs $G_{A_2}$ and $G_{A_3}$ are Figure \ref{fig:ex2ex3}.
\begin{figure}[htbp]
\begin{center}
\unitlength 0.1in
\begin{picture}( 41.9800, 15.6000)( 14.8800,-22.9000)
%
\put(30.3600,-19.4200){\makebox(0,0)[lb]{}}%
%
\put(24.2000,-9.6000){\makebox(0,0)[lb]{}}%
\put(27.0000,-22.3000){\makebox(0,0)[lb]{$3$}}%
%
\special{pn 8}%
\special{pa 4870 1110}%
\special{pa 4840 1100}%
\special{pa 4814 1080}%
\special{pa 4794 1056}%
\special{pa 4776 1030}%
\special{pa 4764 1000}%
\special{pa 4756 970}%
\special{pa 4752 938}%
\special{pa 4752 906}%
\special{pa 4758 874}%
\special{pa 4768 844}%
\special{pa 4782 814}%
\special{pa 4798 788}%
\special{pa 4822 764}%
\special{pa 4848 746}%
\special{pa 4878 734}%
\special{pa 4908 730}%
\special{pa 4940 736}%
\special{pa 4970 746}%
\special{pa 4998 764}%
\special{pa 5020 788}%
\special{pa 5036 814}%
\special{pa 5048 844}%
\special{pa 5058 874}%
\special{pa 5062 906}%
\special{pa 5062 938}%
\special{pa 5058 970}%
\special{pa 5048 1000}%
\special{pa 5036 1030}%
\special{pa 5018 1056}%
\special{pa 4996 1080}%
\special{pa 4970 1098}%
\special{pa 4952 1108}%
\special{sp}%
%
\special{pn 8}%
\special{pa 4956 1106}%
\special{pa 4950 1108}%
\special{fp}%
\special{sh 1}%
\special{pa 4950 1108}%
\special{pa 5018 1092}%
\special{pa 4996 1082}%
\special{pa 4998 1058}%
\special{pa 4950 1108}%
\special{fp}%
%
\special{pn 8}%
\special{pa 1826 2198}%
\special{pa 2612 2188}%
\special{fp}%
\special{sh 1}%
\special{pa 2612 2188}%
\special{pa 2544 2168}%
\special{pa 2558 2188}%
\special{pa 2546 2208}%
\special{pa 2612 2188}%
\special{fp}%
%
\special{pn 8}%
\special{pa 2660 2050}%
\special{pa 2306 1350}%
\special{fp}%
\special{sh 1}%
\special{pa 2306 1350}%
\special{pa 2318 1418}%
\special{pa 2330 1398}%
\special{pa 2354 1400}%
\special{pa 2306 1350}%
\special{fp}%
%
\special{pn 8}%
\special{pa 2564 2110}%
\special{pa 1780 2126}%
\special{fp}%
\special{sh 1}%
\special{pa 1780 2126}%
\special{pa 1848 2144}%
\special{pa 1834 2124}%
\special{pa 1846 2104}%
\special{pa 1780 2126}%
\special{fp}%
\put(15.9000,-22.1000){\makebox(0,0)[lb]{$2$}}%
\put(21.9000,-12.8000){\makebox(0,0)[lb]{$1$}}%
%
\put(56.8600,-19.3200){\makebox(0,0)[lb]{}}%
%
\special{pn 8}%
\special{ar 5438 2180 132 110  0.0000000 6.2831853}%
%
\put(50.7000,-9.5000){\makebox(0,0)[lb]{}}%
\put(53.9000,-22.2000){\makebox(0,0)[lb]{$3$}}%
%
\special{pn 8}%
\special{pa 4390 2000}%
\special{pa 4860 1370}%
\special{fp}%
\special{sh 1}%
\special{pa 4860 1370}%
\special{pa 4804 1412}%
\special{pa 4828 1414}%
\special{pa 4836 1436}%
\special{pa 4860 1370}%
\special{fp}%
%
\special{pn 8}%
\special{pa 4770 1360}%
\special{pa 4306 1992}%
\special{fp}%
\special{sh 1}%
\special{pa 4306 1992}%
\special{pa 4362 1950}%
\special{pa 4338 1950}%
\special{pa 4328 1926}%
\special{pa 4306 1992}%
\special{fp}%
\put(42.6000,-22.1000){\makebox(0,0)[lb]{$2$}}%
\put(48.6000,-12.9000){\makebox(0,0)[lb]{$1$}}%
%
\special{pn 8}%
\special{pa 5410 2010}%
\special{pa 4982 1352}%
\special{fp}%
\special{sh 1}%
\special{pa 4982 1352}%
\special{pa 5002 1418}%
\special{pa 5012 1396}%
\special{pa 5036 1396}%
\special{pa 4982 1352}%
\special{fp}%
%
\special{pn 8}%
\special{pa 4920 1402}%
\special{pa 5352 2056}%
\special{fp}%
\special{sh 1}%
\special{pa 5352 2056}%
\special{pa 5332 1990}%
\special{pa 5324 2012}%
\special{pa 5300 2012}%
\special{pa 5352 2056}%
\special{fp}%
%
\special{pn 8}%
\special{ar 4310 2150 132 110  0.0000000 6.2831853}%
%
\special{pn 8}%
\special{ar 4910 1240 132 110  0.0000000 6.2831853}%
%
\special{pn 8}%
\special{ar 2240 1230 132 110  0.0000000 6.2831853}%
%
\special{pn 8}%
\special{ar 1620 2170 132 110  0.0000000 6.2831853}%
%
\special{pn 8}%
\special{ar 2760 2170 132 110  0.0000000 6.2831853}%
%
\special{pn 8}%
\special{pa 4470 2150}%
\special{pa 5290 2170}%
\special{fp}%
\special{sh 1}%
\special{pa 5290 2170}%
\special{pa 5224 2148}%
\special{pa 5238 2170}%
\special{pa 5224 2188}%
\special{pa 5290 2170}%
\special{fp}%
%
\special{pn 8}%
\special{pa 2170 1380}%
\special{pa 1724 2028}%
\special{fp}%
\special{sh 1}%
\special{pa 1724 2028}%
\special{pa 1778 1984}%
\special{pa 1754 1984}%
\special{pa 1746 1962}%
\special{pa 1724 2028}%
\special{fp}%
\end{picture}%
\end{center}
\caption{The directed graphs $G_{A_2}$ and $G_{A_3}$}
\label{fig:ex2ex3}
\end{figure}
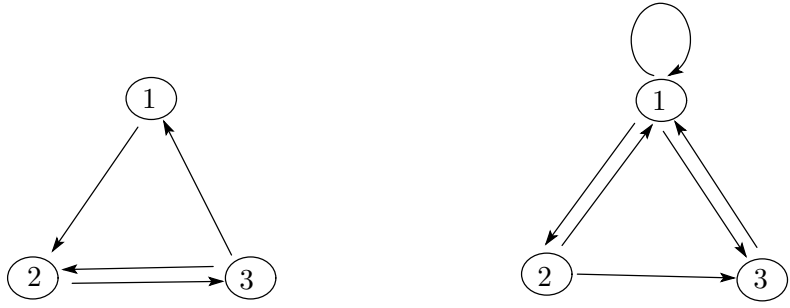
\end{example}

The following lemma shows that a one-sided full shift
has only trivial right Markov code. 
\begin{lemma}\label{lem:fixedpoint}
Let $\C \subset B_*(X_A)$ 
be a right Markov code for $(X_A, \sigma_A)$.
Suppose that $X_A$ has a  fixed point $a^\infty =(a,a,a,\dots ) \in X_A$.
We then have 
$a^m =\overbrace{a\cdots a}^m \in \C$ if and only if $m =1$.
\end{lemma}
\begin{proof}
Let $\C=\{ \omega(1),\omega(2),\dots, \omega(M) \} \subset B_*(X_A)$
be a right Markov code for $(X_A,\sigma_A)$.
By Lemma \ref{lem:factorization},
there uniquely exists
$i_1, i_2, \dots \in \Sigma_{A(\C)}$
such that the fixed point
$a^\infty$ is uniquely factorized as
\begin{equation*}
a^\infty = \omega(i_1)\omega(i_2)\cdots \omega(i_n)\cdots
\end{equation*}
Suppose that 
$$
\omega(i_1) = \overbrace{a\cdots a}^{m_1},
\qquad
\omega(i_2) = \overbrace{a\cdots a}^{m_2},\qquad
\dots
$$
Since the code $\C$ is a prefix code,
we know that $m_1 = m_2 = \cdots $ putting $m_0$.
Suppose that $m_0 >1$.
By the shift invariance property of right Markov code,
there exists $L \in \N$ such that 
 there exist 
$j_1, j_2,\dots, j_k \in \Sigma_{A(\C)}$
 satisfying
\begin{equation} \label{eq:am0}
\sigma_A(a^{m_0} ) \overbrace{a^{m_0} \cdots a^{m_0}}^{L-1}
= \omega(j_1) \omega(j_2) \cdots \omega(j_k).
\end{equation} 
The words $\omega(j_1),  \omega(j_2),  \dots, \omega(j_k)
$ appeared in the right hand side of \eqref{eq:am0}
must be $a^{m_0}$,
because $\C$ is a prefix code.
Hence we see that 
\begin{equation*} 
a^{m_0-1}  \overbrace{a^{m_0} \cdots a^{m_0}}^{L-1}
= \overbrace{a^{m_0}a^{m_0} \cdots a^{m_0}}^k
\end{equation*} 
so that
$m_0 -1 + (L-1)m_0 = k\cdot m_0$,
proving 
$m_0 =1$.
This implies that 
$a^m =\overbrace{a\cdots a}^m \in \C$ if and only if $m =1$.
\end{proof}
Let us denote by $(X_{[N]}, \sigma_{[N]})$ the full $N$-shift over 
$\Sigma_{[N]} =\{1,2,\dots, N \}$
\begin{proposition}
A one-sided full shift
has only trivial right Markov code.
\end{proposition}
\begin{proof}
Let $\C$ be a right Markov code for $(X_{[N]}, \sigma_{[N]})$.
By Lemma \ref{lem:fixedpoint}, we see that 
$\Sigma_{[N]} \subset \C$.
As $\C$ is a prefix code, there are no other words in $\C$ than 
$\Sigma_{[N]}.$
\end{proof}

Let 
$\C=\{ \omega(1),\omega(2),\dots, \omega(M) \} \subset B_*(X_A)$
be a right Markov code for $(X_A,\sigma_A)$.
We write
$\omega(i) = (\omega_1(i),\omega_2(i), \dots,\omega_{\ell(i)}(i)) \in B_{\ell(i)}(X_A)$
with $\omega_j(i) \in \Sigma_A$.
We put
$$
s(\omega(i)) = \omega_{1}(i) \in \Sigma_A, \qquad
r(\omega(i)) = \omega_{\ell(i)}(i) \in \Sigma_A.
$$
Define an $M\times M$ matrix $A(\C)=[A(\C)(i,j)]_{i,j=1}^M$
with entries in $\{0,1\}$ by setting
\begin{equation*}
A(\C)(i,j) = A(r(\omega(i)), s(\omega(j))), \qquad i,j \in \Sigma_{A(\C)}.
\end{equation*}
This means that 
\begin{equation*}
A(\C)(i,j) = 
\begin{cases}
1 & \text{ if } \omega(i) \omega(j) \in B_*(X_A), \\
0 & \text{ otherwise. }
\end{cases}
\end{equation*}

Hence we have a one-sided topological Markov shift
$(X_{A(\C)}, \sigma_{A(\C)})$ 
from a right Markov code $\C$ for $(X_A, \sigma_A)$.
\begin{lemma}
Suppose that $A$ is irreducible and not any permutation.
Then the matrix $A(\C)$ is irreducible and not any permutation.
\end{lemma}
\begin{proof}
By the irreducible condition (iii) of Definition \ref{def:Markovcode},
it is direct to see that the matrix $A(\C)$ is irreducible.
It suffices to show that $A(\C)$ is not any permutation.
As $A$ is irreducible and not any permutation,
for any $x_1 \in \Sigma_A$, 
there exists a finite sequence
$x_2,  \dots, x_k, x_{k+1},x'_{k+1}\in \Sigma_A$
with $x_{k+1} \ne x'_{k+1}$ such that 
$(x_1,x_2, \dots, x_k, x_{k+1}) \in B_{k+1}(X_A)$
and
$(x_1,x_2,  \dots, x_k, x'_{k+1}) \in B_{k+1}(X_A)$.
Extend the admissible words
$(x_1,x_2,  \dots, x_k, x_{k+1})$
and  
$(x_1,x_2,  \dots, x_k, x'_{k+1})$
to its right infinitely as elements of $X_A$.
We denote them by
$x, x' \in X_A$, respectively,
so that 
$x \ne x'$. 
By Lemma \ref{lem:factorization},
there exist $n\in \N$ and 
$$
\omega(i_1)\omega(i_2)\cdots\omega(i_{n-1})\omega(i_n),\qquad
\omega(i_1)\omega(i_2)\cdots\omega(i_{n-1})\omega(i'_n) \in B_*(X_A)
$$
for $x,x'$ so that 
$(i_1, i_2,\dots, i_{n-1}, i_n),
(i_1, i_2,\dots, i_{n-1}, i'_n) \in B_{n}(A(\C)),
$
such that 
$i_n \ne i'_n$.
As $A(\C)$ is irreducible,
any $i_k \in \Sigma_{A(\C)}$
goes to $i_1$,
so that $i_k$ has distinct followers in $X_{A(\C)}$.
This means that 
the matrix $A(\C)$ is not any permutation.
\end{proof}
Two one-sided topological Markov shifts 
$(X_A, \sigma_A)$ and
$(X_B, \sigma_B)$ 
are said to be topologically conjugate if there exists a homeomorphism
$h: X_A\longrightarrow X_B$ such that 
$h\circ\sigma_A = \sigma_B\circ h$.
The homeomorphism $h:X_A\longrightarrow X_B$ is called a topological conjugacy.
Now we will introduce a notion of coded equivalence 
in one-sided topological Markov shifts. 
\begin{definition}
Let $A, B$ be irreducible non permutation  matrices.
The one-sided topological Markov shifts 
$(X_A, \sigma_A)$ and
$(X_B, \sigma_B)$ 
are said to be {\it elementary coded equivalent}\/ 
if there exist a right Markov code $\C_1$ for $(X_A,\sigma_A)$
 and a right Markov code $\C_2$ for $(X_B,\sigma_B)$
such that  
the one-sided topological Markov shifts 
$(X_{A(\C_1)},\sigma_{A(\C_1)})$
and 
$(X_{B(\C_2)},\sigma_{B(\C_2)})$
are topologically conjugate.
It is written 
$
(X_A, \sigma_A) \underset{\cod}{\approx}
(X_B, \sigma_B)$.
If there exists a finite sequence 
$A_0, A_1, A_2,\dots A_m$ of 
irreducible non permutation matrices
such that 
$
A =A_0, A_m =B$
and
$$
(X_{A_0}, \sigma_{A_0}) \underset{\cod}{\approx}
(X_{A_1}, \sigma_{A_1}) \underset{\cod}{\approx}\cdots
                        \underset{\cod}{\approx}(X_{A_{m-1}}, \sigma_{A_{m-1}})
 \underset{\cod}{\approx}(X_{A_m}, \sigma_{A_m}),
 $$
then 
$(X_A, \sigma_A)$ and
$(X_B, \sigma_B)$ 
are said to be {\it coded equivalent}.\/
It is written 
$
(X_A, \sigma_A) \underset{\cod}{\sim}
(X_B, \sigma_B)$.
\end{definition}
\begin{lemma}
\begin{enumerate}
\renewcommand{\theenumi}{\roman{enumi}}
\renewcommand{\labelenumi}{\textup{(\theenumi)}}
\item
Let $A, B$ be irreducible non permutation matrices.
If the one-sided topological Markov shifts 
$(X_A, \sigma_A)$ and
$(X_B, \sigma_B)$ are topologically conjugate,
then they are coded equivalent.
\item
Let 
$\C \subset B_*(X_A)$
be a right Markov code for $(X_A,\sigma_A)$.
Then 
$(X_A,\sigma_A)$ and 
$(X_{A(\C)}, \sigma_{A(\C)})$
are coded equivalent.
\end{enumerate}
\end{lemma}
\begin{proof}
(i) The assertion is clear by considering trivial right Markov codes for each of 
$(X_A, \sigma_A)$ and
$(X_B, \sigma_B)$.

(ii)
 The assertion is also clear by considering trivial right Markov codes for each of 
$(X_A, \sigma_A)$ and
$(X_{A(\C)}, \sigma_{A(\C)})$ .
\end{proof}

We will next introduce a notion of moving block code 
between one-sided topological Markov shifts.
It is a generalization of sliding block codes between 
one-sided topological Markov shifts
and
gives rise to a coded equivalence.

 Let $\C =\{\omega(1), \omega(2), \dots,\omega(M)\} \subset B_*(X_A)$
be a right Markov code for $(X_A,\sigma_A)$.
 Let us define $\kappa_\C: \C\longrightarrow \Sigma_{A(\C)}$
 by $\kappa_\C(\omega(i)) = i$ for $i\in \Sigma_{A(\C)}$,
and a homeomorphism
$h_\C: X_A\longrightarrow X_{A(\C)}$ by setting 
$$
h_\C(\omega(i_1)\omega(i_2)\cdots\omega(i_n)\cdots ) = (i_1, i_2, \dots, i_n, \dots )
\quad
\text{ for }\quad
\omega(i_1)\omega(i_2)\cdots\omega(i_n)\cdots  \in X_A.
$$
We call the homeomorphism
$h_\C: X_A\longrightarrow X_{A(\C)}$
the standard coding homeomorphism and write $h_\C =\kappa_\C^\infty$. 
 
 Let us denote by $B_n(\C)$ the set of concatenated $n$ words of the code $\C$,
 that is 
 $$
 B_n(\C) = \{ \omega(i_1) \cdots \omega(i_n)  \in B_*(X_A) 
 \mid i_1,\dots, i_n \in \Sigma_{A(\C)} \}. 
$$
 \begin{definition}
Let $\C_1 =\{\omega(1), \omega(2), \dots,\omega(M_1)\} \subset B_*(X_A)$
be a right Markov code for $(X_A,\sigma_A)$ and
$\C_2 =\{\xi(1), \xi(2), \dots,\xi(M_2)\} \subset B_*(X_B)$
be a right Markov code for $(X_B,\sigma_B)$.
Let
$\Phi: B_{n+1}(X_{A(\C_1)}) \longrightarrow \Sigma_{B(\C_2)}$
be a block map in the ordinary sense (cf. \cite{LM}).
Then a {\it coded block map}\/ 
$\Phi_\C: B_{n+1}(\C_1) \longrightarrow B_1(\C_2)(= \C_2)$
is defined by
\begin{equation*}
\Phi_C(\omega(i_1)\omega(i_2) \dots \omega(i_{n+1}) )
= \kappa_{\C_2}^{-1}\circ \Phi(i_1,i_2,\dots,i_{n+1}), \qquad
i_1, i_2,\dots,i_{n+1}\in \Sigma_{A(\C_1)}. 
\end{equation*} 
For the block map
$\Phi: B_{n+1}(X_{A(\C_1)}) \longrightarrow \Sigma_{B(\C_2)}$,
let us denote by 
$
\phi :=\Phi_\infty^{[0,n]}: X_{A(\C_1)} \longrightarrow X_{B(\C_2)}
$
the sliding block code with memory $0$ and anticipation $n$ (see \cite[p. 15]{LM}).
The {\it moving block code }\/ 
$$\varphi_{\Phi}: X_A \longrightarrow X_B$$
is a map defined by 
$$
\varphi_\Phi:= h_{\C_2}^{-1} \circ \phi \circ h_{\C_1}.
$$
Hence the diagram
\begin{equation*}
\begin{CD}
X_A @>\varphi_{\Phi}>> X_B \\
@V{h_{\C_1}}VV @VV{h_{\C_2}}V \\
X_{A(\C_1)} @>\phi>> X_{B(\C_2)}
\end{CD}
\end{equation*}
commutes.
\end{definition}
If both right Markov codes $\C_1, \C_2$ are trivial right Markov codes,
then the coded block map and the moving block code 
are block map and sliding block code 
in the ordinary sense.
\begin{lemma}
Let  
$(X_A, \sigma_A)$ and
$(X_B, \sigma_B)$ 
be one-sided topological Markov shifts.
Then they are elementary coded equivalent if and only if 
there exists a homeomorphism
$\varphi_\Phi:X_A\longrightarrow X_B$ 
of moving block code for some block map
$\Phi:B_{n+1}(X_{A(\C_1)}) \longrightarrow \Sigma_{B(\C_2)}$
with right Markov codes
$\C_1$ for 
$(X_A, \sigma_A)$ and
$\C_2$ for
$(X_B, \sigma_B)$,
respectively.
\end{lemma}
\begin{proof}
Since
$\varphi_\Phi:X_A\longrightarrow X_B$ 
is a homeomorphism if and only if 
$
\phi: X_{A(\C_1)} \longrightarrow 
X_{B(\C_2)}
$
is a homeomorphism.
Since 
$
\phi: X_{A(\C_1)} \longrightarrow 
X_{B(\C_2)}
$
is always shift-commuting, we know that  
$\varphi_\Phi:X_A\longrightarrow X_B$ 
is a homeomorphism of moving block code
if and only if 
$
\phi: X_{A(\C_1)} \longrightarrow 
X_{B(\C_2)}
$
is a topological conjugacy.
The latter is equivalent to the condition that 
$(X_A, \sigma_A)$ and
$(X_B, \sigma_B)$ 
 are elementary coded equivalent.
Hence we have the assertion.
\end{proof}
In \cite{MaPacific}, a notion of continuous orbit equivalence of one-sided topological Markov shifts 
was defined in the following way.
Two one-sided topological Markov shifts 
$(X_A, \sigma_A)$ and $(X_B, \sigma_B)$ are said to be continuously orbit equivalent 
if there exist continuous maps
 $k_1,\, l_1:X_A \rightarrow \Zp$
such that
$\sigma_B^{k_1(x)} (h(\sigma_A(x))) = \sigma_B^{l_1(x)}(h(x))$
for $x \in X_A$,
and similarly 
there exist continuous maps
$k_2,\, l_2:X_B \rightarrow \Zp$
such that
$\sigma_A^{k_2(y)} (h^{-1}(\sigma_B(y))) = \sigma_A^{l_2(y)}(h^{-1}(y))$
for $y \in X_B$.
The following proposition is a key  to show Theorem \ref{thm:main}.
\begin{proposition}\label{prop:codedcoe}
Let $A$ be an irreducible, non permutation  matrix with entries in $\{0,1\}$.
Let $\C$ be a right Markov code for $(X_A, \sigma_A)$.
 Then
the one-sided topological Markov shift
$(X_{A(\C)}, \sigma_{A(\C)})$ 
is continuously orbit equivalent to
$(X_A, \sigma_A)$.
\end{proposition}
\begin{proof}
Put
$B = A(\C)$.
By Lemma \ref{lem:factorization} with the unique factorization property of $\C$,
 for any $x =(x_n)_{n\in \N} \in X_A$
there exists a unique increasing sequence
$1<k_1 <k_2  <\cdots  $ of positive integers
and
$i_1, i_2, \cdots \in \Sigma_{B}$ such that 
\begin{equation*}
x_{[1,k_1)} =\omega(i_1), \qquad
x_{[k_1,k_2)} =\omega(i_2),\quad \dots \quad
x_{[k_n,k_{n+1})}=\omega(i_{n+1}), \quad \dots
\end{equation*}
That is,
$x$ has a unique factorization:
\begin{equation*}
x = \omega(i_1)\omega(i_2)\cdots \omega(i_n)\cdots.
\end{equation*}
Let 
$h: X_A \longrightarrow X_{B}$
be the standard coding homeomorphism $h_\C$
defined by 
\begin{equation*}
h(x) = (i_1, i_2, i_3,\dots ) 
\quad \text{ for } \quad x = \omega(i_1)\omega(i_2)\cdots \omega(i_n)\cdots.
\end{equation*}
Since
$$
A(\C)(i_n, i_{n+1}) = A(r(\omega({i_n})), s(\omega({i_{n+1}}))) 
=A(x_{k_n-1}, x_{k_n})=1,
$$
we have $h(x) \in X_{B}.$
It is easy to see that 
$h: X_A\longrightarrow X_{B}$  is continuous.
As $\C$ has unique factorization property, $h$ is bijective,
so that it is a homeomorphism.
Let $U_{\omega(i)}$ be the cylinder set of $X_A$ beginning with the word $\omega(i)$.
As $\C$ is a prefix code with unique factorization property,
we have a disjoint union  
$X_A = \sqcup_{i\in \Sigma_B} U_{\omega(i)}$ 
so that we have a disjoint union
\begin{equation*}
X_A = \bigsqcup_{(i_1,i_2,\dots,i_L) \in B_L(X_{B})} 
U_{\omega(i_1)\omega(i_2)\cdots\omega(i_L)}.
\end{equation*} 
For any $x \in U_{\omega(i_1)\omega(i_2)\cdots\omega(i_L)} \subset X_A$,
we set
$l_1(x) = L, k_1(x) = k(i)$,
where $k(i)$ is the number $k$ determined by \eqref{eq:shiftinvariance}.
As
$\omega(i_1)\omega(i_2)\cdots \omega(i_L) = x_{[1,k_L)},
$
its length is
$ k_L-1$.
We then have
$h(x) = ( i_1,i_2,\dots,i_L)h(x_{[k_L,\infty)})$
so that
\begin{equation*}
\sigma_B^{l_1(x)}(h(x)) = \sigma_B^{L}(h(x)) = h(x_{[k_L,\infty)}).
\end{equation*}
On the other hand,
by \eqref{eq:shiftinvariance}, 
\begin{equation*}
\sigma_A(x) = \sigma_A(\omega(i_1) )\omega(i_2) \cdots \omega(i_L)x_{[k_L,\infty)}
= \omega(j_1) \omega(j_2) \cdots \omega(j_k)x_{[k_L,\infty)}
\end{equation*}
so that 
\begin{equation*}
h(\sigma_A(x)) = (j_1,j_2, \cdots j_k) h(x_{[k_L,\infty)}).
\end{equation*}
We then have 
\begin{equation*}
\sigma_B^{k_1(x)}(h(\sigma_A(x)) 
=\sigma_B^{k}(h(\sigma_A(x))) 
= h(x_{[k_L,\infty)})
\end{equation*}
and hence 
\begin{equation*}
 \sigma_B^{k_1(x)}(h(\sigma_A(x)) =\sigma_B^{l_1(x)}(h(x)).  
\end{equation*}
We will next study 
the inverse $h^{-1}:X_{B} \longrightarrow X_A$.
For $y =(i_1,i_2,i_3,\dots ) \in X_{B}$ with
 $i_1, i_2,i_3, \dots \in \Sigma_B$
consider 
$$
h^{-1}(y) = \omega(i_1)\omega(i_2) \omega(i_3) \cdots  
=(x_1, \dots, x_{k_1-1}, x_{k_1},\dots, x_{k_2-1},x_{k_2},\dots ).
$$
Put
$$
l_2(y) = k_n-1(=\ell(i_n))\quad \text{  for }\quad y \in U_{i_n} \subset X_{B}
$$
Since $X_{B} = \sqcup_{i_j\in \Sigma_B} U_{i_j}$,
the map
$l_2:X_{B}\longrightarrow \Zp$ is continuous map. 
We also put
$k_2(y) =0$ for all $y \in X_{B}$.
As
$\sigma_B(y) = (i_2,i_3,\dots ),
$ 
we have
$$
h^{-1}(\sigma_B(y)) = \omega(i_2) \omega(i_3) \cdots  
=( x_{k_1},\dots, x_{k_2-1},x_{k_2},\dots )
$$
so that 
\begin{align*}
\sigma_A^{l_2(y)}(h^{-1}(y))
& = \sigma_A^{k_1-1}( x_1, \dots, x_{k_1-1},x_{k_1},\dots, x_{k_2-1},x_{k_2},\dots )\\
& = ( x_{k_1},\dots, x_{k_2-1},x_{k_2},\dots )\\
\end{align*}
and hence
we have
\begin{equation*}
\sigma_A^{k_2(y)}(h^{-1}(\sigma_B(y))) =h^{-1}(\sigma_B(y)) =\sigma_A^{l_2(y)}(h^{-1}(y)).
\end{equation*}
We thus  see that 
$(X_A, \sigma_A)$ and
$(X_B, \sigma_B)$ 
are continuously orbit equivalent. 
\end{proof}

The following is a main result of the paper.
\begin{theorem}\label{thm:main}
Let $A$ and $B$ be irreducible non permutation  matrices.
If the one-sided topological Markov shifts 
$(X_A, \sigma_A)$ and
$(X_B, \sigma_B)$ are coded equivalent,
then they are continuously orbit equivalent. 
\end{theorem}
\begin{proof}
We first note that 
continuous orbit equivalence in one-sided topological Markov shifts 
is an equivalence relation (\cite[Theorem 1.1]{MaPacific}).
We may assume that 
$
(X_A, \sigma_A) \underset{\cod}{\approx}
(X_B, \sigma_B)$.
Take a right Markov code $\C_1$(resp. $\C_2$) for $(X_A, \sigma_A)$
(resp. $(X_B, \sigma_B))$ such that  
$(X_{A(\C_1)}, \sigma_{A(\C_1)})$
is topologically conjugate to
$(X_{B(\C_2)}, \sigma_{B(\C_2)})$.
Since
topological conjugacy implies continuous orbit equivalence,
we conclude that 
$(X_A, \sigma_A)$ and
$(X_B, \sigma_B)$ are continuously orbit equivalent,
because of Proposition \ref{prop:codedcoe}. 
\end{proof}

\begin{example}
{\bf 1.}
Let
$A_1=
\begin{bmatrix}
1& 1 \\
1&0
\end{bmatrix}.
$
Define
$\C_1 = \{1, 21\}$.
Then $\C_1$ is a right Markov code.
Put $\omega(1) =1, \omega(2) = 21$.
Since
$$
s(\omega(1)) = r(\omega(1)) =r(\omega(2)) =1,  
\qquad s(\omega(2)) =2,
$$
we have
$
A(\C_1)(i,j) = A(i,j) =1
$
for all $i,j=1,2$
so that
\begin{equation}\label{eq:AC1}
A(\C_1) =
\begin{bmatrix}
1 & 1 \\
1 & 1 
\end{bmatrix}
=[2]
\end{equation}
Therefore 
$
(X_{A_1}, \sigma_{A_1}) \underset{\cod}{\sim}
(X_{[2]}, \sigma_{[2]})$
and hence
$(X_{A_1}, \sigma_{A_1})$ is continuously orbit equivalent to 
the full $2$-shift.
This fact is already seen in \cite{MaPacific}.

{\bf 2.}
Let
$A_2=
\begin{bmatrix}
0 & 1& 0 \\
0 & 0& 1 \\
1 & 1& 0
\end{bmatrix}.
$
Define
$\C_2 = \{12, 23, 323, 312\}$.
Then $\C_2$ is a right Markov code.
Put $\omega(1) =12, \omega(2) = 23, \omega(3) = 323, \omega(4) = 312$.
Since
$$
s(\omega(1)) =1, \quad r(\omega(1)) = r(\omega(4))= s(\omega(2)) =2, \quad
r(\omega(2)) = r(\omega(3))= s(\omega(3))= s(\omega(4)) =3, 
$$
we have
$$
A(\C_2) =
\begin{bmatrix}
0 & 0 & 1 & 1 \\
1 & 1 & 0 & 0 \\
1 & 1 & 0 & 0 \\
0 & 0 & 1 & 1 
\end{bmatrix}.
$$
As the total column amalgamation of the matrix $A(\C_2)$ is the matrix 
$A(\C_1)$ in \eqref{eq:AC1}.
Hence the one-sided topological Markov shift
$(X_{A(\C_2)},\sigma_{A(\C_2)})$
is topologically conjugate to
the full 2-shift
$(X_{[2]},\sigma_{[2]})$.
Theherefore $(X_{A_2},\sigma_{A_2}) 
 \underset{\cod}{\sim}
(X_{[2]}, \sigma_{[2]})$.

{\bf 3.}
Let
$A_3=
\begin{bmatrix}
1 & 1& 1 \\
1 & 0& 1 \\
1 & 0& 0
\end{bmatrix}.
$
Define
$\C_3 = \{1,21, 31, 231\}$.
Then $\C_3$ is a right Markov code.
Put $\omega(1) =1, \omega(2) = 21, \omega(3) = 31, \omega(4) = 231$.
Since
$$
s(\omega(1)) =r(\omega(1)) = r(\omega(2))= r(\omega(3)) =r(\omega(4)) =1, \quad
s(\omega(2)) = s(\omega(4)) =2, \quad
s(\omega(3)) =3,
$$
we have
$$
A(\C_3) =
\begin{bmatrix}
1 & 1 & 1 & 1 \\
1 & 1 & 1 & 1 \\
1 & 1 & 1 & 1 \\
1 & 1 & 1 & 1 
\end{bmatrix}
=[4].
$$
Hence the one-sided topological Markov shift
$(X_{A(\C_3)},\sigma_{A(\C_3)})$
is the full 4-shift $(X_{[4]}, \sigma_{[4]}).$
Theherefore $(X_{A_3},\sigma_{A_3}) 
 \underset{\cod}{\sim}
(X_{[4]}, \sigma_{[4]})$.
  
Related results to classification of Cuntz--Krieger algebras are seen in
\cite{EFW}, \cite{MaYMJ}.

\end{example}

\medskip

{\it Acknowledgments:}
The author would like to thank the referee for his/her careful reading the first draft of the paper.
This work was supported by 
JSPS KAKENHI Grant Numbers 15K04896, 19K03537.



\begin{thebibliography}{99}




\bibitem{Beal}{\sc M-P. ~B\'{e}al},
{\it Codage Symbolique},
Masson, Parris 1993.


\bibitem{BP}{\sc J.~Berstel and D. ~Perrin},
{\it Theory of codes}, Academic Press, London (1985).


\bibitem{BH}
{\sc F. Blanchard and G. Hansel},
{\it Systems cod{\'e}s}, Theor.\ Computer Sci.\ 
{\bf 44}(1986), pp.\ 17--49.


\bibitem{BH2}
{\sc F. Blanchard and G. Hansel},
{\it Sofic constant-to-one extensions of subshifts of finite type}, 
Proc. Amer. Math. Soc. {\bf 112}(1991), pp.\ 259--265.








%
\bibitem{EFW} {\sc M. Enomoto, M. Fujii and Y. Watatani},
{\it $K_0$-groups and classifications of Cuntz--Krieger algebras}, 
Math.\ Japon. {\bf 26}(1981), pp.\ 443--460.

\bibitem{FieD1}
{\sc D. Fiebig},
{\it Common closing extensions and finitary regular isomorphisms for synchronized systems},
in {\it Symbolic Dynamics and Its Applications},
Contemporary Mathematics {\bf 135}(1992), pp. \ 125--138.

\bibitem{FieD2}{\sc D. Fiebig},
{\it Common extensions and hyperbolic factor maps for coded systems},
Ergodic Theory Dynam. Systems  {\bf 35}(1995), pp. \ 517--534.

\bibitem{FieF}
{\sc D. Fiebig and U. -R. Fiebig},
{\it Covers for coded system},
in {\it Symbolic Dynamics and Its Applications},
Contemporary Mathematics {\bf 135}(1992), pp. \ 139--180.


\bibitem{GPS} {\sc T. Giordano, I. F. Putnam and C. F. Skau},
{\it Topological orbit equivalence and $C^*$-crossed products},
J.\ Reine Angew.\ Math.\ {\bf 469}(1995), pp.\ 51--111.



\bibitem{HPS}
{\sc R. H. Herman, I. F. Putnam and C. F. Skau},
{\it Ordered Bratteli diagrams, dimension groups and topological dynamics}, 
Internat. J. Math.
{\bf 3}(1992),pp.\ 827--864.


\bibitem{Kitchens}{\sc B.~P. ~Kitchens},
{\it Symbolic dynamics}, 
Springer-Verlag, Berlin, Heidelberg and New York (1998).



\bibitem{LM}{\sc D. ~Lind and B. ~Marcus},
{\it An introduction to symbolic dynamics and coding},
 Cambridge University Press, Cambridge
(1995).


\bibitem{MaPacific}
{\sc K. Matsumoto},
{\it Orbit equivalence of topological Markov shifts and Cuntz--Krieger algebras},
Pacific J.\ Math.\ 
{\bf 246}(2010), 199--225.



\bibitem{MaYMJ}
{\sc K. Matsumoto},
{\it Some remarks  on orbit equivalence of 
topological Markov shifts and Cuntz--Krieger algebras},
Yokohama Math. J.
{\bf 58}(2012), pp.\ 41--52.















\bibitem{MMKyoto}
{\sc K. Matsumoto and H. Matui},
{\it Continuous orbit equivalence of topological Markov shifts 
and Cuntz--Krieger algebras},
Kyoto J. Math. {\bf 54}(2014), pp.\ 863--878.



\bibitem{MMGGD}
{\sc K. Matsumoto and H. Matui},
{\it Full groups of Cuntz-Krieger algebras and Higman-Thompson groups},
 Groups Geom. Dyn. {\bf 11}(2017), pp. \ 499--531.










\end{thebibliography}
\end{document}